\documentclass[12pt]{amsart}

\usepackage{comment}
\usepackage{color}
\usepackage{amsmath}
\usepackage{amsthm}
\usepackage{amssymb}
\usepackage{graphicx}
\usepackage{enumerate}
\usepackage{amsfonts}
\usepackage{mathrsfs}
\usepackage{parskip}
\usepackage{mathdots}
\usepackage{color}

\numberwithin{equation}{section}
\theoremstyle{plain}
\newtheorem{Proposition}[equation]{Proposition}
\newtheorem{Corollary}[equation]{Corollary}
\newtheorem*{Corollary*}{Corollary}
\newtheorem{Theorem}[equation]{Theorem}
\newtheorem*{Theorem*}{Theorem}
\newtheorem{Lemma}[equation]{Lemma}
\theoremstyle{definition}

\newtheorem{Example}[equation]{Example}
\newtheorem{Remark}[equation]{Remark}
\newtheorem{Question}[equation]{Question}

\allowdisplaybreaks

\def\MM{\mathscr{M}}
\def\Mult{\mathfrak{M}}
\def\P{\mathscr{P}}
\def\A{\mathscr{A}}

\def\C{\mathbb{C}}

\def\D{\mathbb{D}}
\def\T{\mathbb{T}}

\def\phi{\varphi}

\newcommand{\beqa}{\begin{eqnarray*}}
\newcommand{\eeqa}{\end{eqnarray*}}

\renewcommand{\leq}{\leqslant}
\renewcommand{\subset}{\subseteq}

\title[Range spaces]{Multipliers between range spaces of co-analytic Toeplitz operators}

\author[Fricain]{Emmanuel Fricain}
 \address{Laboratoire Paul Painlev\'e, Universit\'e Lille 1, 59 655 Villeneuve d'Ascq C\'edex }
 \email{emmanuel.fricain@math.univ-lille1.fr}

\author[Hartmann]{Andreas Hartmann}
\address{Institut de Math\'ematiques de Bordeaux, Universit\'e Bordeaux 1, 351 cours de la Lib\'eration 33405 Talence C\'edex, France}
\email{Andreas.Hartmann@math.u-bordeaux1.fr}

\author[Ross]{William T. Ross}
	\address{Department of Mathematics and Computer Science, University of Richmond, Richmond, VA 23173, USA}
	\email{wross@richmond.edu}

\keywords{Hardy spaces, inner functions, model spaces, multipliers, de Branges Rovnyak spaces}

\subjclass[2010]{30J05, 30H10, 46E22}

\begin{document}

\begin{abstract}
In this paper we discuss the multipliers between range spaces of co-analytic Toeplitz operators. 
\end{abstract}

\maketitle

\section{Introduction}
This paper continues  the study of the multipliers between two sub-Hardy Hilbert spaces. By the term  ``multiplier'' between two Hilbert spaces $\mathscr{H}_{1}, \mathscr{H}_2$ of analytic functions on the open unit disk $\D = \{|z| < 1\}$ we mean the set 
$\{\phi \in \mathscr{O}(\D):  \forall f\in\mathscr{H}_1,\, \phi  f\in\mathscr{H}_2\}$,
where $\mathscr{O}(\D)$ denotes the analytic functions on $\D$. By `sub-Hardy Hilbert spaces' we mean the Hilbert spaces which can be contractively embedded into the Hardy space $H^2$ of the unit disk. Prominent examples of these types of spaces are the de Branges-Rovnyak spaces \cite{MR3617311, Sa}.

 The study of multipliers between two model spaces began with a paper of Crofoot \cite{Crofoot}  and continued in \cite{FHR-Mult-Model}. This recent work was continued further in \cite{Camara-Part} for two Toeplitz kernels. 
Since model spaces are special examples of de Branges-Rovnyak spaces  \cite{MR3617311, Sa}, it seems natural to expand this investigation to include the multipliers between two de Branges-Rovnyak spaces. The multipliers from a de Branges-Rovnyak space to {\em itself} were discussed in \cite{MR1098860, MR1254125, MR1614726}. In particular, B. Davis and J. McCarthy \cite{MR1098860} obtained a nice description, in terms of growth of Fourier coefficients,  of functions $f\in H^\infty$ (the bounded analytic functions on $\mathbb D$) which multiply every de Branges--Rovnyak space into itself.

This paper explores another type of sub-Hardy Hilbert space, closely connected to de Branges--Rovnyak spaces, namely the range space $\MM(\overline{a}) = T_{\overline{a}} H^2$ of the co-analytic Toeplitz operator $T_{\overline{a}}$ on $H^2$  with symbol $\overline{a}$ where $a\in H^\infty$. 

In \cite{MR1254125} B. Lotto and D. Sarason obtained a characterization of the multipliers of $\MM(\overline{a})$ into itself in terms of the boundedness of the product of two Hankel operators. This characterization is rather difficult to check and one of the aims of this paper is to give a complete functional characterization in some particular situations. It was also observed by Sarason in \cite{MR847333} that any function analytic in a neighborhood of $\overline{\D}$, the closure of $\D$, is a multiplier of every $\MM(\overline{a})$ into itself (see also \cite[Theorem 24.6]{MR3617311}). Since the constant functions belong to $\MM(\overline{a})$, we see that every function analytic in a neighborhood of $\overline{\D}$ (in particular the polynomials) belongs to every $\MM(\overline{a})$.  Note that in \cite{MR3617311, MR847333}, it is assumed that $a$ is an outer function in the closed unit ball of $H^\infty$ that is non extreme (meaning $\log(1-|a|)\in L^1(\mathbb T)$), but since for $a\in H^\infty$,  the function $a_1=a/\lambda$  ($\lambda=2\|a\|_{\infty}$) is non-extreme and  $\MM(\overline{a}) = \MM(\overline{a_1})$ (see Proposition \ref{ppppopOO}), the result of Sarason is true for every $a\in H^\infty$.  In particular, the function $\varphi(z)=z$ multiplies $\MM(\overline{a})$ into itself, which means that the shift operator acts boundedly on $\MM(\overline{a})$ for every $a\in H^\infty$.

To state our results, we set some basic terminology that will be discussed in more detail in the next section. In \cite{FHR-Ma} we described the range space $\MM(\overline{a})$ for {\em rational} $a \in H^{\infty}$. Here one can show (see Proposition \ref{93848bbczvcvgvgvgvgv} below) that if $\zeta_1, \ldots, \zeta_N$ are the zeros of $a$ on the unit circle $\T$, repeated according to their multiplicity, and the polynomial $\check{a}$ is defined by 
\begin{equation}\label{classA}
\check{a}(z) = \prod_{j = 1}^{N} (z - \zeta_j), \quad \zeta_j \in \T,
\end{equation}
 then 
$$\MM(\overline{a}) = \MM(\overline{\check{a}}).$$
In \cite[Cor.~6.21]{FHR-Ma} we identified $\MM(\overline{\check{a}})$ as 
\begin{equation}\label{4938oryehdfgfe}
\MM(\overline{\check{a}}) = \check{a} H^2 \dotplus \P_{N - 1},
\end{equation}
where $\mathscr{P}_{N - 1}$ denotes the polynomials of degree at most $N - 1$ and $\dotplus$ denotes the algebraic direct sum (not necessarily an orthogonal sum). We denote the polynomials  of the form \eqref{classA} by $\mathscr{A}$. 

For $a_1, a_2 \in H^\infty$, let 
$$\Mult(\overline{a_1}, \overline{a_2}) := \{\phi \in \mathscr{O}(\D):  \forall f\in \MM(\overline{a_1}),\, \phi f\in  \MM(\overline{a_2})\}$$ denote the set of {\em multipliers} from $\MM(\overline{a_1})$ to $\MM(\overline{a_2})$. When $a_1 = a_2 = a$, we set
$$\Mult(\overline{a}) := \Mult(\overline{a}, \overline{a})$$ 
 to be the multipliers from $\MM(\overline{a})$ to itself. 
From standard theory of reproducing kernel Hilbert spaces of analytic functions on $\mathbb D$, one can show that multipliers (from a space to itself)  must be a subset of $H^\infty$. Moreover, since the constant functions belong to $\MM(\overline{a})$, the multiplier space $\Mult(\overline{a})$ is always contained in  $\MM(\overline{a}) \cap H^{\infty}$. In Proposition \ref{8sd2lsewh} we prove that when $a \in \mathscr{A}$, the two sets coincide,  that is
\begin{equation}\label{nnnnnnncc}
 \Mult(\overline{a}) = \MM(\overline{a}) \cap H^{\infty}.
 \end{equation}
 Since $\Mult(\overline{a})$ is an algebra, \eqref{nnnnnnncc} shows that, at least for $a \in \mathscr{A}$, the set $\MM(\overline{a}) \cap H^{\infty}$ is an algebra. It is worth mentioning here that, for general $a \in H^{\infty}$,  Lotto and Sarason \cite{MR1614726} proved that $\MM(\overline{a}) \cap H^{\infty}$ is not always al algebra and thus some additional conditions on $a$ need to be imposed.

 For more general $a \in H^{\infty}$ this is not always the case (see Remark \ref{q98yrgeouiwergfh}).

Two of the main theorems of this paper are complete descriptions of  $\Mult(\overline{a_1}, \overline{a_2})$ for certain $a_1, a_2 \in \mathscr{A}$. For instance, when $a_1/a_2 \in H^{\infty}$, in other words, the zero set of $a_2$ is contained in the zero set of $a_1$ (counting multiplicity), then $\MM(\overline{a_1}) \subset \MM(\overline{a_2})$ (Proposition \ref{ppppopOO}) and we have the following:

\begin{Theorem}\label{ydayd1818347}
Suppose that $a_1, a_2 \in \A$ and $h = a_1/a_2 \in H^{\infty}$. Then 
$$\Mult(\overline{a_1}, \overline{a_2}) = \{\phi \in \MM(\overline{a_2}): h \phi \in H^{\infty}\}.$$
\end{Theorem}

When the division is reversed, i.e., $a_2/a_1 \in H^{\infty}$, in other words, the zero set of $a_1$ is contained in the zero set of $a_2$ (counting multiplicity), then $\MM(\overline{a_2}) \subset \MM(\overline{a_1})$ and we have the following:

\begin{Theorem}\label{yyysatta6666}
Suppose $a_1, a_2 \in \A$ with $k:= a_2/a_1 \in H^{\infty}$. Then 
$$\Mult(\overline{a_1}, \overline{a_2}) = k (\MM(\overline{a_1}) \cap H^{\infty}).$$
\end{Theorem}

When $a \in \mathscr{A}$ (and $\|a\|_{\infty} \leq 1$) then $\MM(\overline{a}) = \mathcal{H}(b)$, where $\mathcal{H}(b)$ is the de Branges-Rovnyak space corresponding to $b$ and $b$ is the {\em Pythagorean mate} for $a$: the unique outer function in $H^{\infty}$ such that $b(0)>0$ and $|a|^2 + |b|^2 = 1$ on $\T$. It turns out that since $a$ is a rational function (in fact a polynomial), then $b$ will also be a rational function. So, our description of $\Mult(\overline{a_1}, \overline{a_2})$ for certain $a_1, a_2 \in \mathscr{A}$, yields a description of the multipliers between the de Branges-Rovnyak spaces $\mathcal{H}(b_1)$ and $\mathcal{H}(b_2)$ for the corresponding Pythagorean mates $b_1$ and $b_2$. 

We have already noticed that the special function $\phi(z) = z$ multiplies $\MM(\overline{a})$ to itself for every $a \in H^{\infty}$. Our next result computes the norm of the multiplication operator $f \mapsto z f$ on $\MM(\overline{a})$. 

\begin{Theorem}\label{10w74hs-}
 If $S f = z f$ is the unilateral shift on $H^2$, then, for any outer function $a\in H^{\infty}$, $S \MM(\overline{a}) \subset \MM(\overline{a})$ and the norm of $S_{\overline{a}} := S|_{\MM(\overline{a})}$
 is equal to
$$ \frac{1}{|a(0)|} \left(\int_{0}^{2 \pi} |a(e^{i \theta})|^2 \frac{d \theta}{2 \pi}\right)^{1/2}.$$
\end{Theorem}

Finally, we explore, as was done in other multiplier space and range space settings \cite{MR1098860, MR1065054}, which functions belong to {\em all} of the multiplier spaces $\Mult(\overline{a})$. 
%

\begin{Theorem}\label{Thm:multiplier-everyMabar}
$$\bigcap_{a \in H^{\infty} \setminus \{0\}} \Mult(\overline{a}) = \mathscr{F},$$
where $\mathscr{F}$ is the set of $\psi\in H^\infty$ whose Fourier coefficients satisfy
$$
\widehat{\psi}(n)=O(e^{-c \sqrt{n}}), \quad n \geqslant 0,$$
for some $c>0$. 
\end{Theorem}
Note that when $\widehat{\psi}(n)=O(e^{-cn})$, then $\psi$ is analytic on a neighborhood of $\overline{\mathbb D}$ (Hadamard's formula for the radius of convergence of a power series) and then, by the result of Sarason mentioned above, $\psi$ is a multiplier of every $\MM(\overline{a})$. Thus Theorem~\ref{Thm:multiplier-everyMabar} is an improvement of this fact.

\section{Basic fact about range spaces}

For $a \in H^{\infty}$ and outer,  the co-analytic Toeplitz operator $T_{\overline{a}}$ \cite{Bottcher} on the Hardy space $H^2$ \cite{Duren}  is injective (Since $a$ is outer, the analytic Toeplitz operator $T_{a} = T_{\overline{a}}^{*}$ has dense range). One can define the range space as 
$$\MM(\overline{a}) := T_{\overline{a}} H^2$$ and, since $T_{\overline{a}}$ is injective, endow $\MM(\overline{a})$ with the range norm $\|\cdot\|_{\overline{a}}$ defined by 
\begin{equation}\label{Tzznorm}
\|T_{\overline{a}} f\|_{\overline{a}} := \|f\|_{H^2} = \left(\int_{0}^{2 \pi} |f(e^{i \theta})|^2 \frac{d \theta}{2 \pi}\right)^{1/2}.
\end{equation}
In the above, we are norming, in the standard way, $H^2$ functions via their radial $L^2 = L^2(d \theta/2 \pi)$ boundary values on $\T$ \cite[p.~21]{Duren}. One can show  \cite[Corollary 3.4]{FHR-Ma} that 
$\MM(\overline{a})$ is a reproducing kernel Hilbert space with kernel function 
\begin{equation}\label{1snsdrgnnzz9}
k_{\lambda}^{\overline{a}} = T_{\overline{a}} (a k_{\lambda}),\qquad \lambda\in\mathbb D,
\end{equation}
where 
\begin{equation}\label{CK}
 k_{\lambda}(z) = \frac{1}{1 - \overline{\lambda} z},\qquad z\in\mathbb D,
 \end{equation}
  is the standard reproducing kernel (the Cauchy kernel) for $H^2$. By the term ``reproducing kernel''  for $\MM(\overline{a})$, we mean 
$$\langle f, k_{\lambda}^{\overline{a}}\rangle_{\overline{a}} = f(\lambda), \quad f \in \MM(\overline{a}),  \lambda \in \D,$$
where $\langle \cdot, \cdot\rangle_{\overline{a}}$ is the inner product arising from the Hilbert space norm in \eqref{Tzznorm}. From time to time we will need the corresponding inner product on $H^2$ which we will denote by $\langle \cdot, \cdot\rangle_{H^2}$.

Observe that when $a\in\mathscr{A}$, \eqref{4938oryehdfgfe} implies that $aH^2\subset\MM(\overline{a})$. In fact, this is true for any function $a\in H^\infty$ since
 $$T_{a}f=T_{\overline{a}}T_{a/\overline{a}}f, \quad f\in H^2.$$ 

Let us complete this preliminary section by showing how to reduce the problem of describing the multiplier space $\Mult(\overline{a_1}, \overline{a_2})$, for rational $a_1, a_2 \in H^{\infty}$, to that of $\check{a_1}, \check{a_2} \in \mathscr{A}$ described in \eqref{classA}. Basic theory of Hardy spaces \cite[p.~24]{Duren} says that every $a \in H^{\infty}$ can be factored as $a = u a_0$, where $u \in H^{\infty}$ is inner and $a_0 \in H^{\infty}$ is outer. Using the Douglas factorization lemma \cite[p.~2]{Sa}, one can prove the following two results. 

\begin{Proposition}{\cite[Lemma 17.3]{MR3617311}}\label{yyyyfgggfggf}
If $a \in H^{\infty}$ and $a_0$ is its outer factor, then $\MM(\overline{a}) = \MM(\overline{a_0})$.
\end{Proposition}

\begin{Proposition}{\cite[Lemma 17.5]{MR3617311}}\label{ppppopOO}
If $a_1, a_2 \in H^{\infty}$ and outer then
$$\frac{a_1}{a_2} \in H^{\infty} \iff \MM(\overline{a_1}) \subset \MM(\overline{a_2}).$$
\end{Proposition}

Our final reduction from rational  $H^{\infty}$ functions to the class $\mathscr{A}$ comes from applying the previous two propositions. We leave the details to the reader. 

\begin{Proposition}\label{93848bbczvcvgvgvgvgv}
Suppose $a \in H^{\infty}$ and rational and let $\check{a}$ be the polynomial defined by 
$$\check{a}(z) = \prod_{j = 1}^{N} (z - \zeta_j),$$
where $\zeta_1, \ldots, \zeta_N$ are the zeros of $a$ on $\T$, repeated according to multiplicity. 
Then $\MM(\overline{a}) = \MM(\overline{\check{a}})$. 
\end{Proposition}

\section{Multiplier spaces}

As already noticed, from the general theory of reproducing kernel Hilbert spaces, a multiplier from such a space to {\em itself} must be a bounded function and thus $\Mult(\overline{a}) \subset H^{\infty}$. Since the constant functions belong to $\MM(\overline{a})$ we see that $\Mult(\overline{a}) \subset H^{\infty} \cap \MM(\overline{a})$. For functions $a\in \A$, we have equality. Note that this fact was already observed by Sarason \cite{MR847333} in the special case when $a(z)=(1-z)/2$ (see also \cite[Corollary 28.29]{MR3617311}). 

\begin{Proposition}\label{8sd2lsewh}
Suppose $a \in \A$. Then 
$\Mult(\overline{a}) = \MM(\overline{a}) \cap H^{\infty}.$
\end{Proposition}

\begin{proof}
From the previous paragraph, we have the $\subset$ containment. 

Now suppose that $\phi \in \MM(\overline{a}) \cap H^{\infty}$. By \eqref{4938oryehdfgfe} 
$$\phi = a \widetilde{\phi} + p, \quad \widetilde{\phi} \in H^2, p \in \mathscr{P}_{N - 1}, N = \mbox{deg}(a).$$
This implies that $a \widetilde{\phi} = \phi - p \in H^{\infty}$. If $f \in \MM(\overline{a})$ then, again by \eqref{4938oryehdfgfe}, 
$$f = a \widetilde{f} + q, \quad \widetilde{f} \in H^2, q \in \mathscr{P}_{N - 1},$$ and so 
$$\phi f = (a \widetilde{\phi} + p) ( a \widetilde{f} + q) = a (a \widetilde{\phi} \widetilde{f} + p \widetilde{f}  + q \widetilde{\phi}) + p q.$$
We have already shown that $a \widetilde{\phi} \in H^{\infty}$ and thus $a \widetilde{\phi} \widetilde{f} \in H^2$. Clearly the terms $p \widetilde{f}$ and $q \widetilde{\phi}$ belong to $H^2$ and so, using \eqref{4938oryehdfgfe}, 
$$a (\widetilde{\phi} \widetilde{f} + p \widetilde{f}  + q \widetilde{\phi}) \in a H^2 \subset \MM(\overline{a}).$$
Since $\MM(\overline{a})$ contains the polynomials, we have $p q \in \MM(\overline{a})$ and thus $\phi f \in \MM(\overline{a})$. Hence $\phi \in \Mult(\overline{a})$ and the $\supseteq$ inclusion follows. 
\end{proof}

\begin{Remark}\label{q98yrgeouiwergfh}
\begin{enumerate}
\item Proposition \ref{8sd2lsewh} says that for $a \in \mathscr{A}$, the set $\MM(\overline{a})\cap H^\infty$ is an algebra (since it is equal to the multiplier algebra $\Mult(\overline{a})$). For general $a \in H^{\infty}$ we do not always have $\Mult(\overline{a}) = \MM(\overline{a}) \cap H^{\infty}$  since there are $a \in H^{\infty}$  such that $\MM(\overline{a})\cap H^{\infty}$ is {\em not} an algebra \cite{MR1614726}.
 \item For a general bounded outer function $a$ we have $H^{\infty} \subset \Mult(\overline{a})$  if and only if the Toeplitz  operator $T_{a/\bar a}$ is invertible \cite[Theorem 17.20]{MR3617311}. In this case we in fact have $\MM(a)=\MM(\overline{a})$ and $\Mult(\overline{a})=H^\infty$. 
\item In \cite{MR1254125} they obtained the following characterization of $\Mult(\overline{a})$ for a general bounded outer function $a$: Let $\varphi\in \MM(\overline{a})\cap H^\infty$ and let $\psi\in H^2$ such that $\varphi=T_{\bar a}\psi$. Then the following are equivalent: (i) $\varphi \in \Mult(\overline{a})$; (ii) the operator $H^*_{\bar\psi}H_{\bar a}$ is bounded on $H^2$ ($H_{\overline{\psi}}$ and $H_{\overline{a}}$ are Hankel operators). This result is difficult to check, even when $a \in \mathscr{A}$, which makes Proposition \ref{8sd2lsewh} all the more useful.
\end{enumerate}
\end{Remark}

The rest of this section contains the proofs of Theorem \ref{ydayd1818347} and Theorem \ref{yyysatta6666}. First observe that $\Mult(\overline{a_1}, \overline{a_2})$ is never trivial. 

\begin{Proposition}
For any $a_1, a_2 \in \A$, $a_2 H^{\infty} \subset \Mult(\overline{a_1}, \overline{a_2}).$
\end{Proposition}

\begin{proof}
Let $\phi \in H^{\infty}$. Then for any $f \in \MM(\overline{a_1})$ we can use \eqref{4938oryehdfgfe} to see that 
$$f =  a_1 \widetilde{f} + p, \quad \widetilde{f} \in H^2, p \in \mathscr{P}_{N_1 - 1}, N_1 = \mbox{deg}(a_1).$$
Thus 
$$(a_2 \phi) f = a_2 \phi (a_1 \widetilde{f} + p) = a_2 (a_1 \phi \widetilde{f} + \phi p) \in a_2 H^2 \subset \MM(\overline{a_2}). \qedhere$$
\end{proof}


\begin{proof}[Proof of Theorem \ref{ydayd1818347}]
We remind that  $1 \in \MM(\overline{a_1})$. This means that if $\phi \in \Mult(\overline{a_1}, \overline{a_2})$ then $\phi \MM(\overline{a_1}) \subset \MM(\overline{a_2})$ and so $\phi =  \phi \cdot 1 \in \MM(\overline{a_2})$. By 
\eqref{4938oryehdfgfe},
\begin{equation}\label{cxncx8asmsd0dhyy}
\phi = a_2 \widetilde{\phi} + p, \quad \widetilde{\phi} \in H^2, p \in \mathscr{P}_{N_2 - 1}, N_2 = \mbox{deg}(a_2).
\end{equation}
With $h := a_1/a_2$, which we assume to belong to $H^{\infty}$, observe that 
$$a_1 \widetilde{\phi} = \frac{\phi - p}{a_2} a_1 = h (\phi - p)$$ and so 
$$a_1 \widetilde{\phi} \in H^{\infty} \iff h \phi \in H^{\infty}.$$ Thus 
$$\{\phi \in \MM(\overline{a_2}): h \phi \in H^{\infty}\} = \{\phi \in \MM(\overline{a_2}): \widetilde{\phi} a_1 \in H^{\infty}\}.$$
To complete the proof, we will now prove that 
$$\Mult(\overline{a_1}, \overline{a_2}) = \{\phi \in \MM(\overline{a_2}): \widetilde{\phi} a_1 \in H^{\infty}\}.$$

($\subset$): Let $\phi \in \Mult(\overline{a_1}, \overline{a_2})$ and recall from \eqref{cxncx8asmsd0dhyy} that 
$\phi = a_2 \widetilde{\phi} + p$. We will show that $a_1 \widetilde{\phi} \in H^{\infty}$ by showing that $a_1 \widetilde{\phi}$ is a multiplier from $H^2$ to itself. Indeed let $\widetilde{f} \in H^2$ and define $f = a_1 \widetilde{f}$. From \eqref{4938oryehdfgfe}, $f$ belongs to $\MM(\overline{a_1})$. By our assumption that $\phi$ is a multiplier, $\phi f \in \MM(\overline{a_2})$. Moreover, 
$$\phi f = (a_2 \widetilde{\phi} + p) (a_1 \widetilde{f}) = a_1 a_2 \widetilde{\phi} \widetilde{f} + p a_1 \widetilde{f}.$$
For the second summand above, observe that 
$$p a_1 \widetilde{f}  = p a_2 h \widetilde{f}  \in a_2 H^2 \subset \MM(\overline{a_2}).$$
Since $\phi f \in \MM(\overline{a_2})$ by assumption, it must be the case that the first summand, i.e., $a_1 a_2 \widetilde{\phi} \widetilde{f}$ belongs to  $\MM(\overline{a_2})$ and so 
$$a_1 a_2 \widetilde{\phi} \widetilde{f} = a_2 \widetilde{F} + R, \quad \widetilde{F} \in H^2, R \in \mathscr{P}_{N_2 - 1}.$$
This means that 
$$a_1 \widetilde{\phi} \widetilde{f} = \widetilde{F} + \frac{R}{a_2}.$$
Clearly $\widetilde{F} \in H^2$ and  $a_1 \widetilde{\phi} \widetilde{f} \in H^1$. Thus, since $R/a_2 \in H^1$ and rational, it follows that $R/a_2 \in H^{\infty}$. 
In summary,
 $$a_1 \widetilde{\phi} \widetilde{f} \in H^2, \quad \widetilde{f} \in H^2,$$ which makes $a_1 \widetilde{\phi}$ a multiplier of $H^2$ and hence bounded. 
 
 ($\supseteq$): Let $f \in \MM(\overline{a_1})$. Then by \eqref{4938oryehdfgfe} 
 $$f = a_1 \widetilde{f} + p, \quad \widetilde{f} \in H^2, p \in \mathscr{P}_{N_1 - 1}, N_1 = \mbox{deg}(a_1).$$
  If $\phi= a_2 \widetilde{\phi} + q \in \MM(\overline{a_2})$ with $a_1 \widetilde{\phi} \in H^{\infty}$ then  
$$\phi f = (a_2 \widetilde{\phi} + q) (a_1 \widetilde{f} + p) = a_2 (a_1 \widetilde{\phi} \widetilde{f} + \widetilde{\phi} p + h \widetilde{f} q) + q p.$$
For the first summand above, observe that, by assumption $a_1 \widetilde{\phi} \in H^{\infty}$, and also that 
$\widetilde{\phi}p$ and $h \widetilde{f} q$ belong to $H^2$ and so the first summand belongs to $a_2 H^2 \subset \MM(\overline{a_2})$. Since $\MM(\overline{a_2})$ contains all the polynomials, $pq \in \MM(\overline{a_2})$. Thus $\phi \in \Mult(\overline{a_1}, \overline{a_2})$. 
\end{proof}

\begin{Corollary}
For $a \in \A$, 
$\Mult(\overline{a}, 1) = \{\phi \in H^2: a \phi \in H^{\infty}\}.$
\end{Corollary}

Notice the above characterizes the multipliers from $\MM(\overline{a})$ (the smaller space) to $H^2$ (the bigger space). 


From Proposition \ref{8sd2lsewh}, the multipliers from $\MM(\overline{a})$ to {\em itself} must be bounded functions. However, as the following example shows, for {\em different} $a_1, a_2 \in \mathscr{A}$, it is possible for $\Mult(\overline{a_1}, \overline{a_2})$ to contain unbounded functions.  

\begin{Example}
If $a_1(z) = (1 + z) (1 - z)$ and $a_2(z) = (1 + z)$, then the unbounded function 
$$\phi(z) = (1 + z)^{1/2 + \epsilon} (1 - z)^{-1/2 + \epsilon}, \quad \epsilon \in (0, \tfrac{1}{2}),$$ 
belongs to $\Mult(\overline{a_1}, \overline{a_2})$. 
To see this observe that $\phi \in a_2 H^2 \subset \MM(\overline{a_2})$ and 
$$\frac{a_1}{a_2} \phi = (1 - z) \phi = (1 + z)^{1/2 + \epsilon} (1 - z)^{1/2 + \epsilon} \in H^{\infty}.$$
Now apply Theorem \ref{ydayd1818347}.

The proof of Theorem \ref{yyysatta6666} is more involved and needs a little bit more set-up. For 
$$a = \prod_{j = 1}^{N} (z - \zeta_j), \quad \zeta_j \in \T, $$
we have seen from \eqref{4938oryehdfgfe} the decomposition
$$\MM(\overline{a}) =  a H^2 \dotplus \P_{N - 1}.$$
Since $\P_{N-1}$ is a finite dimensional space, it follows that $aH^2$ is a closed subspace of $\MM(\overline{a})$ and standard functional analysis arguments (see for instance \cite[Theorem 5.16]{MR1157815}) show that the projection with range $aH^2$ and null space $\P_{N-1}$ is continuous and we have
\begin{equation}\label{xxx22625qwrlkncv<<}
\|a f\|_{\overline{a}} \asymp \|f\|_{H^2}, \quad f \in H^2.
\end{equation}
With this set up we are now ready to prove Theorem \ref{yyysatta6666}.

\end{Example}



\begin{proof}[Proof of Theorem \ref{yyysatta6666}]
By assumption, $k = a_2/a_1 \in H^{\infty}$ We first show that  $k \in \Mult(\overline{a_1}, \overline{a_2})$. Indeed if $f = a_1 \widetilde{f} + p \in \MM(\overline{a_1})$, then 
$$k f = \frac{a_2}{a_1} (a_1 \widetilde{f} + p) = a_2 \widetilde{f} + p \frac{a_2}{a_1}.$$
The first term  belongs to $a_2 H^2 \subset \MM(\overline{a_2})$ while the second term is analytic in a neighborhood of $\overline{\D}$ (since it is a rational function in $H^2$) and hence belongs to $\MM(\overline{a_2})$. Thus $k \in \Mult(\overline{a_1}, \overline{a_2})$.

Next observe, from the inclusion $\MM(\overline{a_2}) \subset \MM(\overline{a_1})$ that 
\begin{equation}\label{wewfjsdf77373}
\Mult(\overline{a_1}, \overline{a_2}) \subset \Mult(\overline{a_1}).
\end{equation} We are now ready to prove 
$$\Mult(\overline{a_1}, \overline{a_2}) = k (\MM(\overline{a_1}) \cap H^{\infty}).$$

($\supseteq$): Let $\phi \in \MM(\overline{a_1}) \cap H^{\infty}$ and $f \in \MM(\overline{a_1})$. By Proposition \ref{8sd2lsewh}, $\phi \in \Mult(\overline{a_1}, \overline{a_1})$ and thus $\phi f \in \MM(\overline{a_1})$. We argued before that $k \in \Mult(\overline{a_1}, \overline{a_2})$ and so $k \phi f \in \MM(\overline{a_2})$. Hence $k \phi \in \Mult(\overline{a_1}, \overline{a_2})$. 

($\subset$): Let $\phi \in \Mult(\overline{a_1}, \overline{a_2})$. Since $1 \in \MM(\overline{a_1})$ we see that $\phi \in \MM(\overline{a_2})$ and hence 
\begin{equation}\label{oooiiiuwuwuuw}
\phi = a_2 \widetilde{\phi} + p = k a_1 \widetilde{\phi} + p, \quad \mbox{deg}(p) \leq N_2 - 1.
\end{equation}
 Our first step is to show that $\phi \in k \MM(\overline{a_1})$ and to do this, we need to show that $p/k$ is a polynomial. To this end, let $h\in H^2$ and put $f=a_1h$. Then 
\begin{equation}\label{0009987}
\phi f = a_2 a_1 \widetilde{\phi} h+ p a_1 h.
\end{equation}
 However, $\phi f \in \MM(\overline{a_2})$ (since $\phi \in \Mult(\overline{a_1}, \overline{a_2})$ and $f\in a_1H^2\subset\MM(\overline{a_1})$) and so 
 \begin{equation}\label{886d6sd6dsf6}
\phi f = a_2 \widetilde{\phi f} + t, \quad \mbox{deg} (t) \leq N_2 - 1
\end{equation}
for some $\widetilde{\phi f}  \in H^2$. 
A calculation using \eqref{0009987} and \eqref{886d6sd6dsf6} yields 
$$
a_1 p h=a_2(\widetilde{\phi f}-a_1\widetilde{\phi}h)+t.
$$
Let $g = \widetilde{\phi f} - a_1 \widetilde{\phi} h $ and observe that $g \in H^1$. Using the division algorithm for polynomials we get 
$$t = a_1 \gamma + \delta, \quad \mbox{deg}(\delta) \leq N_1 - 1.$$
This yields 
\begin{align*}
a_1 h p  = a_2 g + t
 = a_2 g + a_1 \gamma + \delta,
\end{align*}
that is 
$$
a_1hp=a_1(kg+\gamma)+\delta.
$$
Now observe that 
\begin{align*}
k g  = k \widetilde{\phi f} - a_1 k \widetilde{\phi}h
 = k \widetilde{\phi f} - a_2 \widetilde{\phi}h
 = k  \widetilde{\phi f} - (\phi - p) h
\end{align*}
which belongs to $H^2$ since $\phi \in H^{\infty}$ (recall \eqref{wewfjsdf77373}).  By the uniqueness of the representation \eqref{4938oryehdfgfe} and the equality $a_1hp=a_1(kg+\gamma)+\delta$, we conclude that $\delta = 0$ and $h p = k g + \gamma$. Finally recall that $g \in H^1$ and so 
$$|g(r \xi)| \lesssim \frac{1}{1 - r}.$$ Hence the function $k g + \gamma$ has a radial limit at each zero of $k$, along with its derivatives of order one less than the order of the zero of $k$. Thus the same must be true for $h p$. But $h$ was an arbitrary element of $H^2$. This means that $p$ must have a zero at every zero of $k$ of at least the multiplicity of the zero of $k$. Conclusion: $p/k$ is a polynomial $b$. 

From the above and the representation of $\phi$ from \eqref{oooiiiuwuwuuw} we know that 
$$\phi = k (a_1 \widetilde{\phi} + b), \quad \mbox{deg}(b) \leq N_1 - 1.$$
Notice in the last step how we used that $\mbox{deg}(p) \leq N_2 - 1$ and $\mbox{deg}(k) = N_2 - N_1$. Note that 
$$\phi_0 := a_1 \widetilde{\phi} + b \in \MM(\overline{a_1})$$ and so it remains to show that $\phi_0 \in H^{\infty}$. 
Since $\phi \in \Mult(\overline{a_1}, \overline{a_2})$ we know that 
\begin{equation}\label{pppqq11455}
\|\phi(a_1 h)\|_{\overline{a_2}} \lesssim \|a_1 h\|_{\overline{a_1}}, \quad h \in H^{\infty}.
\end{equation}
However, 
$$\phi a_1 h = k a_1 \phi_0 h = a_2 h \phi_0.$$
From \eqref{xxx22625qwrlkncv<<} we have 
$$\|a_1 h\|_{\overline{a_1}} \asymp \|h\|_{H^2}, \quad \|a_2 h \phi_0\|_{\overline{a_2}} \asymp \|h \phi_0\|_{H^2}.$$ Combining this with \eqref{pppqq11455} we get 
$$\|h \phi_0\|_{H^2} \lesssim \|h\|_{H^2}, \quad h \in H^{\infty}.$$ This means that the operator 
$h \mapsto \phi_0 h$, initially defined on $H^{\infty} \subset H^2$, extends to a bounded multiplication operator on $H^2$. It is well known that the multipliers of $H^2$ must be bounded and so $\phi_0 \in H^{\infty}$ which completes the proof. 
\end{proof}

\begin{Corollary}
For $a \in \A$, 
$\Mult(1, \overline{a}) = a H^{\infty}$.
\end{Corollary}

Notice how the above corollary characterizes the multipliers between $H^2$ (the bigger space) and $\MM(\overline{a})$ (the smaller space). 

\section{Onto multipliers}

Crofoot \cite{Crofoot} studied the onto multipliers between model spaces. Here we discuss the onto multipliers between $\MM(\overline{a_1})$ and $\MM(\overline{a_2})$, i.e., $\phi \in \mathscr{O}(\D)$ for which $\phi \MM(\overline{a_1}) = \MM(\overline{a_2})$. 

\begin{Theorem}\label{ppsdfusd7sdfbbvxxz}
Suppose $a_1, a_2 \in \mathscr{A}$ with $a_2/a_1 = h \in H^{\infty} \setminus \C$. Then there are no multipliers from $\MM(\overline{a_1})$ onto $\MM(\overline{a_2})$. 
\end{Theorem}

\begin{proof}
Suppose there is a $\phi \in H^2$ with $\phi \MM(\overline{a_1}) = \MM(\overline{a_2})$. By Theorem \ref{yyysatta6666} there is a $\psi \in H^{\infty}$ so that $\phi = h \psi$. But since $1 \in \MM(\overline{a_2})$ there is a $g \in \MM(\overline{a_1})$ such that $1 = \phi g$. Thus $1/\phi \in H^2$ and hence $\psi/\phi  = 1/h \in H^2$. However, $1/h$ is a non-constant rational function with poles on $\T$ and thus can not belong to $H^2$ -- which yields a contradiction. 
\end{proof}

\begin{Corollary}\label{bsudfysdfiusyf}
Suppose $a_1, a_2 \in \mathscr{A}$ with $a_1/a_2 \in H^{\infty} \setminus \C$. Then there are no multipliers from $\MM(\overline{a_1})$ onto $\MM(\overline{a_2})$. 
\end{Corollary}

\begin{proof}
Suppose there is a $\phi \in H^2$ with $\phi \MM(\overline{a_1}) = \MM(\overline{a_2})$. Then 
$\frac{1}{\phi} \MM(\overline{a_2}) = \MM(\overline{a_1})$.
Apply Theorem \ref{ppsdfusd7sdfbbvxxz} to obtain a contradiction. 
\end{proof}

When $a_1 = a_2$, there are indeed plenty of onto multipliers.  

\begin{Proposition}
If $a \in H^\infty$ and $\lambda \in \mathbb C$, $|\lambda|<\|a\|_{\infty}^{-1}$, then 
$$\frac{1}{1 - \overline{\lambda} a} \MM(\overline{a}) = \MM(\overline{a}).$$
\end{Proposition}

\begin{proof}
Fix $\lambda \in \D$, $|\lambda|<\|a\|_{\infty}^{-1}$, and $a \in H^\infty$ and note that for any $f \in \MM(\overline{a})$ we have the identity 
$$\frac{f}{1 - \overline{\lambda} a} = f + a \frac{\overline{\lambda} f}{1 - \overline{\lambda} a}$$ which belongs to $\MM(\overline{a})$ since the first term belongs to $\MM(\overline{a})$ and the second belongs to $a H^2$ which is contained in $\MM(\overline{a})$. Thus 
$$\frac{1}{1 - \overline{\lambda} a} \in \Mult(\overline{a}, \overline{a}).$$

On the other hand for $f \in \MM(\overline{a})$, we have 
$$(1 - \overline{\lambda} a) f = f - a (\overline{\lambda} f) \in \MM(\overline{a})$$
for similar reasons as before. Thus $(1 - \overline{\lambda} a) \in \Mult(\overline{a}, \overline{a})$ which completes the proof. 
\end{proof}

\begin{Question}
Is there a tractable description of all of the onto multipliers from $\MM(\overline{a})$ to itself? 
\end{Question}

\section{Intersections of multiplier spaces}
In this section, we prove Theorem~\ref{Thm:multiplier-everyMabar}. Recall that $\mathscr{F}$ denotes the set of $\psi \in H^{\infty}$ whose Fourier coefficients 
$$\widehat{\psi}(n) = \int_{0}^{2 \pi} \psi(e^{i  \theta}) e^{- i n \theta} \frac{d \theta}{2 \pi}$$
satisfy 
$$\widehat{\psi}(n) = O(e^{-c \sqrt{n}}), \quad n \geqslant 0,$$
for some $c > 0$ and let 
$$\mathscr{B} = \{\phi \in H^{\infty} \setminus\{0\}: \|\phi\|_{\infty} \leq 1, \log(1 - |\phi|) \in L^1\}$$ denote the (non-zero)  non-extreme points in the closed unit ball of $H^{\infty}$. Also recall that 
$$\Mult(\mathscr{H}(b)) = \{\phi \in \operatorname{Hol}(\D): \phi \mathscr{H}(b) \subset \mathscr{H}(b)\}$$ are the multipliers of $\mathscr{H}(b)$ to itself. 

In \cite{MR1098860} it was shown that 
 \begin{equation}\label{oosdpfisdf11}
\bigcap_{b \in \mathscr{B}}\Mult(\mathscr{H}(b) )= \mathscr{F}
\end{equation}
and in \cite{MR1065054} it was shown that 
\begin{equation}\label{pptttfffffzzz2}
\bigcap_{\phi \in H^{\infty}\setminus\{0\}} \MM(\overline{\phi}) = \mathscr{F}.
\end{equation}

\begin{proof}[Proof of Theorem \ref{Thm:multiplier-everyMabar}]
From \cite[Thm.~20.17]{MR3617311} we know that if $b \in \mathscr{B}$ and $a$ is the so-called ``Pythagorean mate'' for $b$, meaning the unique $a \in \mathscr{B}$ with $a(0) > 0$ and with $|a|^2 + |b|^2 = 1$ almost everywhere on $\T$ (such mates exist by standard Hardy space theory \cite{Duren} and the assumption that $b$ is non-extreme)
\begin{equation}\label{11nsdve5rtyghjfb}
\Mult(\mathscr{H}(b)) \subset \Mult(\overline{a}).
\end{equation}
Moreover, for any $a \in H^{\infty} \setminus \{0\}$ with $\|a\|_{\infty} \leq 1$ (not necessarily non-extreme) we have 
$$
\MM(\overline{a}) = \MM(\overline{a/2})
$$
and, more importantly, $a/2 \in \mathscr{B}$. This all yields 
\begin{equation}\label{rye89wiofdp}
\bigcap_{a \in H^{\infty} \setminus \{0\}} \MM(\overline{a}) = \bigcap_{a \in \mathscr{B}} \MM(\overline{a}).
\end{equation}
Our preliminary comment is that $\MM(\overline{a})$ always contains the constant functions and thus
\begin{equation}\label{7765432}
\Mult(\overline{a}) \subset \MM(\overline{a}).
\end{equation}
Putting this all together we have 
\begin{align*}
\mathscr{F} & = \bigcap_{b \in \mathscr{B}} \Mult(\mathscr{H}(b)) && \mbox{(by \eqref{oosdpfisdf11})}\\
&  \subset \bigcap_{a \in \mathscr{B}} \Mult(\overline{a}) && \mbox{(by \eqref{11nsdve5rtyghjfb})}\\
& \subset \bigcap_{a \in \mathscr{B}} \MM(\overline{a}) && \mbox{(by \eqref{7765432})}\\
& \subset \bigcap_{a \in H^{\infty} \setminus \{0\}} \MM(\overline{a}) && \mbox{(by \eqref{rye89wiofdp})}\\
& = \mathscr{F} && \mbox{(by \eqref{pptttfffffzzz2})}.
\end{align*}
This completes the proof. 
\end{proof}

\section{The commutant and the norm of the shift on $\MM(\overline{a})$}

We know that the identity function $\phi(z)  = z$ belongs to $\Mult(\overline{a})$. 
This means that if $S f = z f$ is the standard unilateral shift on $H^2$, then the operator 
$$S_{\overline{a}} := S|_{\MM(\overline{a})}$$ is a well defined bounded operator on $\MM(\overline{a})$. The next result, which is quite standard for the shift operator on many Hilbert spaces of analytic functions, computes the commutant  of $S_{\overline{a}}$. If 
$\mathscr{B}(\MM(\overline{a}))$ denotes the  bounded operators on $\MM(\overline{a})$, the commutant $\{S_{\overline{a}}\}'$ is defined to be 
$$\{S_{\overline{a}}\}' := \{A \in \mathscr{B}(\MM(\overline{a})): A S_{\overline{a}} = S_{\overline{a}} A\}.$$

\begin{Proposition}
For $a \in H^{\infty}$ and outer,  
$$\{S_{\overline{a}}\}' = \{M_{\phi}: \phi \in \Mult(\overline{a})\},$$
where $M_{\phi}$ is the multiplication operator $M_{\phi} f = \phi f$ on $\MM(\overline{a})$. 
\end{Proposition}

\begin{proof}
Clearly we have $\supseteq$. To prove the other containment, let 
 $A \in \mathscr{B}(\MM(\overline{a}))$ with $A S_{\overline{a}} = S_{\overline{a}} A$. This implies that for any polynomial $p$
 $$A (p(S_{\overline{a}}) 1) = p(S_{\overline{a}}) A(1),$$ 
 equivalently, $A(p) = p A(1)$. Since the polynomials are dense in $\MM(\overline{a})$ (see \cite[Theorem 17.4]{MR3617311}), for a given $f \in \MM(\overline{a})$ we can find a sequence of polynomials $\{p_n\}_{n \geqslant 1}$ such that $p_{n} \to f$ in the norm of $\MM(\overline{a})$. Since point evaluations on $\D$ are continuous in the norm of $\MM(\overline{a})$ (indeed $\MM(\overline{a})$ is a reproducing kernel Hilbert space), we see that 
 $p_n \to f$ pointwise on $\D$. 
 Since $A p_n \to A f$ both in norm as well as pointwise on $\D$, we see that 
 $A f = A(1) f$. Thus $A(1)$ is a multiplier of $\MM(\overline{a})$ and $A  = M_{A(1)}$. 
\end{proof}

\begin{Remark}
 The above fact is quite standard for many spaces of analytic functions (see also \cite[Theorem 9.16]{FM} for a broader setting).
\end{Remark}

Making adjustments to the above proof, we can show that 
$$\{A \in \mathscr{B}(\MM(\overline{a_1}), \MM(\overline{a_2})): A S_{\overline{a_1}} = S_{\overline{a_2}} A\} = \{M_{\phi}: \phi \in \Mult(\overline{a_1}, \overline{a_2})\}.$$
In the above, $M_{\phi}: \MM(\overline{a_1}) \to \MM(\overline{a_2})$, $M_{\phi} f = \phi f$. One can also show that this set is (operator) norm closed. 

In the rest of this section we prove Theorem \ref{10w74hs-}.  For $f \in \mathscr{O}(\D)$ let 
$$B f = \frac{f - f(0)}{z}$$ denote the backward shift of $f$. It is well known that 
$$T_{\overline{z}} f = Bf, \quad f \in H^2,$$ and that $B$ acts contractively on $H^2$. Furthermore, since 
$T_{\overline{z}} T_{\overline{a}} = T_{\overline{a}} T_{\overline{z}}$ we conclude that 
$B \MM(\overline{a}) \subset \MM(\overline{a})$. This allows us to define 
$$X_{\overline{a}} := B|_{\MM(\overline{a})}.$$ Observe that $X_{\overline{a}}$ is also a contraction since for any $f = T_{\overline{a}} g \in \MM(\overline{a})$ we have 
$$\|X_{\overline{a}} f\|_{\overline{a}} = \|T_{\overline{a}} T_{\overline{z}} g\|_{\overline{a}} = \|T_{\overline{z}} g\|_{H^2} \leq \|g\|_{H^2} = \|f\|_{\overline{a}}.$$ 


\begin{Proposition}\label{Thm:adjoint-Xbara}
Let $a$ be a bounded outer function. Then
$$
X_{\overline a}^*=S_{\overline{a}}+1\otimes_{\overline{a}}T_{\overline{a}}Ba.
$$
\end{Proposition}

\begin{proof}
For $f=T_{\overline a}g\in\MM(\overline a)$ and  $\lambda\in\D$, we have
\[
(X_{\overline a}^* f)(\lambda)=\langle X_{\overline a}^*f,k_{\lambda}^{\overline a}\rangle_{\overline a}=\langle f,X_{\overline a}k_\lambda^{\overline a}\rangle_{\overline a}.
\]
Let $P_{+}$ denote the orthogonal projection of $L^2$ onto $H^2$ and $P_{-} = \operatorname{Id} - P_{+}$. 
Using the definition of $X_{\overline{a}}$ along with the identities 
$$k_{\lambda}^{\overline{a}}=T_{\overline{a}}a
k_{\lambda}=T_{|a|^2}k_{\lambda},$$
and $T_{\overline{a}} T_{\overline{z}} = T_{\overline{z}} T_{\overline{a}}$, we obtain 
$$X_{\overline a}k_\lambda^{\overline a}=BT_{\overline a}(ak_\lambda)=T_{\overline a}B(ak_\lambda).$$ 
From here we get 
\begin{align*}
(X_{\overline a}^* f)(\lambda) &=\langle f,T_{\overline{a}} B(a k_{\lambda})\rangle_{\overline a}\\
&=\langle g,B(ak_\lambda)\rangle_{L^2}\\
&=\langle g,\overline zak_\lambda \rangle_{L^2}\\
&=\langle \overline ag,\overline z k_\lambda\rangle_{L^2}\\
&=\langle P_+(\overline ag),\overline z k_\lambda\rangle_{L^2}+\langle \overline ag,P_-(\overline z k_\lambda)\rangle_{L^2}\\
&=\langle zf,k_\lambda \rangle_{L^2}+\langle \overline ag,P_-(\overline z k_\lambda)\rangle_{L^2}\\
&=\lambda f(\lambda)+\langle \overline ag,P_-(\overline z k_\lambda)\rangle_{L^2}.
\end{align*}
A short computation with power series shows that $P_-(\overline zk_\lambda)=\overline z$, whence 
\[
(X_{\overline a}^* f)(\lambda)=\lambda f(\lambda)+\langle \overline ag,\overline z\rangle_{L^2}.
\]
Observe now 
\begin{align*}
\langle \overline ag,\overline z\rangle_{L^2}&= \langle g,P_{+}(a\bar z)\rangle_{L^2}\\
&=\langle g,Ba\rangle_{L^2}\\
&=\langle T_{\overline{a}}g,T_{\overline{a}}Ba\rangle_{\overline{a}}\\
&=\langle f,T_{\overline{a}}Ba\rangle_{\overline{a}}\\
&=(1\otimes_{\overline{a}} T_{\overline{a}}Ba)f.
\end{align*}
Hence $X_{\overline a}^* f=S_{\overline{a}}f+(1\otimes T_{\overline{a}}Ba)f$, which yields the result.
\end{proof}

%
%
%

\begin{proof}[Proof of Theorem \ref{10w74hs-}]
First, observe from Proposition~\ref{Thm:adjoint-Xbara} that $S_{\overline{a}}=X_{\overline{a}}^*-1\otimes_{\overline{a}}T_{\overline{a}}Ba$ and so $S_{\overline{a}}^*=X_{\overline{a}}-T_{\overline{a}}Ba\otimes_{\overline{a}}1$. Thus for every $f\in\MM(\overline{a})$, we have
\begin{align*}
S_{\overline{a}}^{*} S_{\overline{a}} f & =X_{\overline{a}}S_{\overline{a}}f-\langle S_{\overline{a}}f,1 \rangle_{\overline{a}}T_{\overline{a}}Ba
=f-\langle S_{\overline{a}}f,1 \rangle_{\overline{a}}T_{\overline{a}}Ba.
\end{align*}
Next we see that 
\begin{align*}
\|S_{\overline{a}} f\|_{\overline{a}}^{2} & = \langle S_{\overline{a}}^{*} S_{\overline{a}}f,f \rangle_{\overline{a}}\\
&=\|f\|_{\overline{a}}^2-\langle S_{\overline{a}}f,1 \rangle_{\overline{a}} \langle T_{\overline{a}}Ba,f \rangle_{\overline{a}}\\
&=\|f\|_{\overline{a}}^2-\langle f,S_{\overline{a}}^*1 \rangle_{\overline{a}} \langle T_{\overline{a}}Ba,f \rangle_{\overline{a}}
\end{align*}
But $S_{\overline{a}}^*1=X_{\overline{a}}1-\|1\|_{\overline{a}}^2T_{\overline{a}}Ba=-\|1\|_{\overline{a}}^2T_{\overline{a}}Ba$, which yields
\begin{align*}
\|S_{\overline{a}} f\|_{\overline{a}}^{2}&=\|f\|^2_{\overline{a}}+\|1\|_{\overline{a}}^2\,|\langle f,T_{\overline{a}}Ba \rangle|^2\\
&\leq (1 + \|1\|_{\overline{a}}^2 \|T_{\overline{a}} B a\|_{\overline{a}}^{2}) \|f\|_{\overline{a}}^{2}.
\end{align*}
This proves the upper bound 
$$\|S_{\overline{a}}\|_{\overline{a}}^{2} \leq 1 + \|1\|_{\overline{a}}^2 \|T_{\overline{a}} B a\|_{\overline{a}}^{2}.$$
To obtain equality, observe that the previous computation shows that 
$$\|S_{\overline{a}} f\|_{\overline{a}}^{2} = \|f\|_{\overline{a}}^{2} + \|1\|_{\overline{a}}^{2} \left| \langle f, T_{\overline{a}} B a\rangle_{\overline{a}}\right|^2.$$
Applying this identity to $f = T_{\overline{a}} B a$ we get 
\begin{align*}
\|S_{\overline{a}} (T_{\overline{a}} B a)\|_{\overline{a}}^{2} & = \|T_{\overline{a}} B a\|_{\overline{a}}^{2} + \|1\|_{\overline{a}}^{2} \|T_{\overline{a}} B a\|_{\overline{a}}^{4}\\
& = \|T_{\overline{a}} B a\|_{\overline{a}}^{2} (1 + \|1\|_{\overline{a}}^{2} \|T_{\overline{a}} B a\|_{\overline{a}}^{2})
\end{align*}
and thus 
$$\|S_{\overline{a}}\|^{2} = 1 + \|1\|_{\overline{a}}^{2} \|T_{\overline{a}} B a\|_{\overline{a}}^{2}.$$
Also observe that 
$$\|1\|_{\overline{a}}^{2} = \Big\|T_{\overline{a}} \frac{1}{\overline{a(0)}}\Big\|_{\overline{a}}^{2} = \Big\|\frac{1}{\overline{a(0)}}\Big\|_{H^2}^{2} = \frac{1}{|a(0)|^2}$$ and 
\begin{align*}
\|T_{\overline{a}} B a\|_{\overline{a}}^{2}  & = \|B a\|_{H^2}^{2}
 = \langle S S^{*} a, a\rangle_{H^2}\\
& = \langle a - a(0), a\rangle_{H^2}\\
& = \|a\|_{H^2}^{2} - |a(0)|^2,
\end{align*}
to conclude 
\begin{align*}
\|S_{\overline{a}}\|^{2} & = 1 + \|1\|_{\overline{a}}^{2} \|T_{\overline{a}} B a\|_{\overline{a}}^{2}\\
& = 1 + \frac{1}{|a(0)|^2} (\|a\|_{H^2}^{2} - |a(0)|^2)
 = \frac{\|a\|_{H^2}^{2}}{|a(0)|^2}. \qedhere
\end{align*}
\end{proof}

\bibliographystyle{plain}

\bibliography{references}

\end{document}